\documentclass[a4paper,12pt]{amsart}

\usepackage{amsmath,amssymb,amscd}

\title{Reparametrizations of vector fields and their shift maps}
\author{Sergiy Maksymenko}
\address{Topology dept., Institute of Mathematics of NAS of Ukraine, Te\-re\-shchenkivska st. 3, Kyiv, 01601 Ukraine}
\email{maks@imath.kiev.ua}

\theoremstyle{plain}
\newtheorem{theorem}[subsection]{Theorem}
\newtheorem{lemma}[subsection]{Lemma}

\newtheorem{cor}[subsection]{Corollary}

\newtheorem{rem}[subsection]{Remark}

\theoremstyle{definition}

\begin{document}

\begin{abstract}
Let $M$ be a smooth manifold, $F$ be a smooth vector field on $M$, and $(\mathbf{F}_t)$ be the local flow of $F$.
Denote by $Sh(F)$ the subset of $C^{\infty}(M,M)$ consisting of maps $h:M \to M$ of the following form:
$$h(x) = \mathbf{F}_{\alpha(x)}(x),$$
where $\alpha$ runs over all smooth functions $M\to\mathbb{R}$ which can be substituted into $\mathbf{F}$ instead of $t$.
This space often contains the identity component of the group of diffeomorphisms preserving orbits of $F$.
In this note it is shown that $Sh(F)$ is not changed under reparametrizations of $F$, that is for any smooth strictly positive function $\mu:M\to(0,+\infty)$ we have that $Sh(F)=Sh(\mu F)$.
As an application it is proved that $F$ can be reparametrized to induce a circle action on $M$ if and only if there exists a smooth function $\mu:M\to(0,+\infty)$ such that $\mathbf{F}(x,\mu(x))\equiv x$.
\end{abstract}

\maketitle

\newcommand\AFld{F}
\newcommand\AFlow{\mathbf{F}}
\newcommand\BFld{G}
\newcommand\BFlow{\mathbf{G}}
\newcommand\GFlow{\mathcal{G}}

\newcommand\funcA{\mathsf{func}(\AFld)}
\newcommand\domA{\mathsf{dom}(\AFlow)}
\newcommand\funcB{\mathsf{func}(\BFld)}
\newcommand\domB{\mathsf{dom}(\BFlow)}

\newcommand\imSh[1]{Sh(#1)}

\newcommand\ShA{\varphi}
\newcommand\imShA{\imSh{\AFld}}
\newcommand\ShB{\psi}
\newcommand\imShB{\imSh{\BFld}}

\newcommand\kerA{\ker(\ShA)}
\newcommand\kerB{\ker(\ShB)}

\newcommand\Per{\mathrm{Per}}

\newcommand\orb{o}
\newcommand\Mman{M}
\newcommand\RRR{\mathbb{R}}
\newcommand\NNN{\mathbb{N}}
\newcommand\ZZZ{\mathbb{Z}}

\newcommand\singA{\Sigma_{\AFld}}
\newcommand\singB{\Sigma_{\BFld}}

\newcommand\Uman{U}

\newcommand\afunc{\alpha}
\newcommand\bfunc{\beta}
\newcommand\mufunc{\mu}

\newcommand\ImFld[1]{\mathrm{im}(#1)}
\newcommand\ImFldk[2]{\mathrm{im}(#1,#2)}

\newcommand\ImF{\ImFld{\AFld}}
\newcommand\ImG{\ImFld{\BFld}}

\newcommand\ImFinf{\ImFldk{\AFld,\infty}}
\newcommand\ImGinf{\ImFldk{\BFld,\infty}}

\newcommand\amap{f}
\newcommand\id{\mathrm{id}}

\newcommand\sact{\Gamma}
\newcommand\ract{\BFlow}
\newcommand\fldact{\BFld}

\newcommand\contSS[2]{\mathsf{S}^{#1,#2}}
\newcommand\Sr[1]{\mathsf{S}^{#1}}

\newcommand\dif{f}
\newcommand\gdif{g}

\section{Introduction}
Let $\Mman$ be a smooth manifold and $\AFld$ be a smooth vector field on $\Mman$ tangent to $\partial\Mman$.
For each $x\in\Mman$ its \emph{integral trajectory} with respect to $\AFld$ is a unique mapping $\orb_x: \RRR\supset(a_x,b_x) \to \Mman$ such that $\orb_x(0)=x$ and $\frac{d}{dt}\orb_x = \AFld(\orb_x)$, where $(a_x,b_x) \subset\RRR$ is the maximal interval on which a map with the previous two properties can be defined.
The image of $\orb_{x}$ will be denoted by the same symbol $\orb_{x}$ and also called the \emph{orbit} of $x$.
It follows that from standard theorems in ODE the following subset of $\Mman\times\RRR$
$$
\domA = \mathop\cup\limits_{x\in\Mman} x \times (a_x, b_x),
$$
is an open, connected neighbourhood of $\Mman\times0$ in $\Mman\times\RRR$.
Then the \emph{local flow\/} of $\AFld$ is the following map
$$\AFlow:\Mman\times \RRR \supset \domA \to \Mman,
\qquad
\AFlow(x,t) = \AFlow_x(t).
$$
It is well known that if $\Mman$ is compact, or $\AFld$ has compact support, then $\AFlow$ is defined on all of $\Mman$.

Denote by $\funcA\subset C^{\infty}(\Mman,\RRR)$ the subset consisting of functions $\afunc:\Mman\to\RRR$ whose graph $\Gamma_{\afunc}=\{(x,\afunc(x)) : x\in\Mman\}$ is contained in $\domA$.
Then we can define  the following map
$$
\ShA:C^{\infty}(\Mman,\RRR) \supset \funcA \longrightarrow C^{\infty}(\Mman,\Mman),
$$
$$
\ShA(\afunc)(x)=\AFlow(x,\afunc(x)).
$$
This map will be called the \emph{shift map} along orbits of $\AFld$ and its image in $C^{\infty}(\Mman,\Mman)$ will be denoted by $\imShA$.

It is easy to see, \cite[Lm.~2]{Maks:Shifts}, that $\ShA$ is $\contSS{r}{r}$-continuous for all $r\geq0$, that is continuous between the corresponding $\Sr{r}$ Whitney topologies of $\funcA$ and $C^{\infty}(\Mman,\Mman)$.

Moreover, if the set $\singA$ of singular points of $\AFld$ is nowhere dense, then $\ShA$ is locally injective, \cite[Pr.~14]{Maks:Shifts}.
Therefore it is natural to know whether it is a homeomorphism with respect to some Whitney topologies, and, in particular, whether it is $\contSS{r}{s}$-open, i.e.\! open as a map from $\Sr{r}$ topology of $\funcA$ into $\Sr{s}$ topology of the image $\imShA$, for some $r,s\geq0$.
These problems and their applications were treated e.g. in~\cite{Maks:Shifts,Maks:AGAG:2006,Maks:LocInvSh}.

In this note we prove the following theorems describing the behaviour of the image of shift maps under reparametrizations and pushforwards.

\begin{theorem}\label{th:image_of_shift_map}
Let $\mufunc:\Mman\to\RRR$ be any smooth function and $\BFld=\mufunc\AFld$ be the vector field obtained by the multiplication $\AFld$ by $\mufunc$.
Then
\begin{equation}\label{equ:imShB_in_imShA}
\imSh{\BFld} \subset \imShA.
\end{equation}

Suppose that $\mufunc\not=0$ on all of $\Mman$.
Then $$\imSh{\mufunc\AFld} = \imShA.$$

In this case the shift mapping $\ShA:\funcA\to\imShA$ of $\AFld$ is $\contSS{r}{s}$-open for some $r,s\geq0$, if and only if so is the shift mapping $\ShB:\funcB\to\imShB$ of $\BFld$.
\end{theorem}

\begin{theorem}\label{th:induced_vf}
Let $z\in\Mman$, $\afunc:(\Mman,z)\to\RRR$ be a germ of smooth function at $z$, and $\dif:\Mman\to\Mman$ be a germ of smooth map defined by $\dif(x)=\AFlow(x,\afunc(x))$.
Suppose that $\dif$ is a germ of diffeomorphism at $z$.
Then
\begin{equation}\label{equ:hF_1FaF}
\dif_{*}\AFld=(1+\AFld(\afunc))\cdot \AFld,
\end{equation}
where $\dif_{*}\AFld = T\dif\circ\AFld\circ\dif^{-1}$ is the vector field induced by $\dif$, and $\AFld(\afunc)$ is the derivative of $\afunc$ along $\AFld$.
Thus $\dif_{*}\AFld$ is just a reparametrization of $\AFld$.

If $\afunc:\Mman\to\RRR$ is defined on all of $\Mman$ and $\dif=\ShA(\afunc)$ is a diffeomorphism of $\Mman$, then
$$\imSh{\dif_{*}\AFld} = \imShA.$$
\end{theorem}

Further in \S\ref{sect:periodic_sh_maps} we will apply these results to circle actions.
In particular, we prove that $\AFld$ can be reparametrized to induce a circle action on $\Mman$ if and only if there exists a smooth function $\mu:\Mman\to(0,+\infty)$ such that $\AFlow(x,\mu(x))\equiv x$, see Corollary~\ref{cor:persh--cicrleact}.

\section{Proofs of Theorems~\ref{th:image_of_shift_map} and~\ref{th:induced_vf}}
These theorems are based on the following well-known statement, see e.g.~\cite{Totoki:MFCKUS:1966,Parry:JLMS:1972,OrnsteinSmorodsky:IJM:1978} for its variants in the category of measurable maps.
\begin{lemma}\label{lm:reparam}
Let $\BFld=\mufunc\AFld$ and $\BFlow:\domB \to \Mman$ be the local flow of $\BFld$.
Then there exists a smooth function $\afunc:\domB\to\RRR$ such that 
$$
\BFlow(x,s)=\AFlow(x,\afunc(x,s)).
$$
In fact,
\begin{equation}\label{equ:afunc-def}
\afunc(x,s)=\int\limits_{0}^{s} \mufunc(\BFlow(x,\tau)) d\tau.
\end{equation}
In particular, for each $\gamma\in\funcB$ we have that 
\begin{equation}\label{equ:Gg_Fag}
\BFlow(x,\gamma(x))=\AFlow(x,\afunc(x,\gamma(x))),
\end{equation}
whence $\imShB\subset\imShA$.
\end{lemma}
\begin{proof}
Put $\GFlow(x,s) = \AFlow(x,\afunc(x,s))$, where $\afunc$ is defined by~\eqref{equ:afunc-def}.
We have to show that $\BFlow=\GFlow$.

Notice that a flow $\BFlow$ of a vector field $\BFld$ is a \emph{unique\/} mapping that satisfies the following ODE with initial condition:
$$
\left.\frac{\partial\BFlow(x,s)}{\partial s}\right|_{s=0} = \BFld(x) = \AFld(x) \mufunc(x), \qquad \BFlow(x,0)=x.
$$
Notice that 
$$
\afunc(x,0)=0,
\qquad
\afunc'_{s}(x,0)=\mu(\BFlow(x,0))=\mu(x).
$$
In particular, $\GFlow(x,0) = \AFlow(x,\afunc(x,0)) = x$.
Therefore it remains to verify that 
\begin{equation}\label{equ:dGFlow}
\left.\frac{\partial\GFlow(x,s)}{\partial s}\right|_{s=0} = \AFld(x) \cdot \mufunc(x).
\end{equation}
We have:
\begin{equation}\label{equ:dGFlow_1} 
\frac{\partial\GFlow}{\partial s}(x,s) = 
\frac{\partial\AFlow}{\partial s}(x,\afunc(x,s)) = 
\left.\frac{\partial\AFlow(x,t)}{\partial t}\right|_{t=\afunc(x,s)} \cdot  \afunc'_{s}(x,s).
\end{equation}
Substituting $s=0$ in~\eqref{equ:dGFlow_1} we get~\eqref{equ:dGFlow}.
\end{proof}

\subsection*{Proof of Theorem~\ref{th:image_of_shift_map}}
Eq.~\eqref{equ:imShB_in_imShA} is established in Lemma~\ref{lm:reparam}.

Suppose that $\mu\not=0$ on all of $\Mman$.
Then $\AFld=\frac{1}{\mu}\BFld$, and $\frac{1}{\mu}$ is smooth on all of $\Mman$.
Hence again by Lemma~\ref{lm:reparam} $\imShA\subset\imShB$, and thus $\imShA=\imShB$.

To prove the last statement define a map $\xi:\funcB\to\funcA$ by
$$
\xi(\gamma)(x)=
\afunc(x,\gamma(x))=\int\limits_{0}^{s} \mufunc(\BFlow(x,\tau)) d\tau,
\qquad \gamma\in\funcB.
$$
Then~\eqref{equ:Gg_Fag} means that the following diagram is commutative:
$$
\begin{CD}
\funcB @>{\xi}>>  \funcA \\
@V{\ShB}VV @VV{\ShA}V  \\
\imShB @=   \imShA 
\end{CD}
$$ 
We claim that $\xi$ is a homeomorphism with respect to $\Sr{r}$ topologies for all $r\geq0$.
Indeed, evidently $\xi$ is $\contSS{r}{r}$-continuous.
Put 
\begin{equation}\label{equ:beta_xi_inv}
\bfunc(x,s)=\int\limits_{0}^{s} \frac{d\tau}{\mufunc(\AFlow(x,\tau))}.
\end{equation}
Then the inverse map $\xi^{-1}:\funcA\to\funcB$ is given by 
\begin{equation}\label{equ:xi_inv}
\xi^{-1}(\delta)(x)=\bfunc(x,\delta(x))=\int\limits_{0}^{\delta(x)} \frac{d\tau}{\mufunc(\AFlow(x,\tau))},
\qquad \delta\in\funcA,
\end{equation}
and is also $\contSS{r}{r}$-continuous.
Hence $\ShB$ is $\contSS{r}{s}$-open iff so is $\ShA$.
Theorem~\ref{th:image_of_shift_map} is completed.

\subsection*{Proof of Theorem~\ref{th:induced_vf}.}
First we reduce the situation to the case $\afunc(z)=0$.
Suppose that $a=\afunc(z)\not=0$ and let $\bfunc(x)=\afunc(x)-a$.
Define the following germ of diffeomorphisms $\gdif=\AFlow_{-a}\circ \dif$ at $z$:
$$
\gdif(x)=
\AFlow(\AFlow(x,\afunc(x)),-a) =\AFlow(x,\afunc(x)-a)=
\AFlow(x,\bfunc(x)).
$$
Then $\gdif(z)=z$, and $\bfunc(z)=0$.

Since $\AFlow$ preserves $\AFld$, i.e. $(\AFlow_t)_{*}\AFld=\AFld$ for all $t\in\RRR$, we obtain that
$$\dif_{*}\AFld = \dif_{*}(\AFlow_{-a})_{*}\AFld = (\dif\circ\AFlow_{-a})_{*}\AFld=\gdif_{*}\AFld.$$
Moreover, $\AFld(\afunc)=\AFld(\bfunc)$.
Therefore it suffices to prove our statement for $\gdif$.

If $z$ is a singular point of $\AFld$, i.e. $\AFld=0$, then both parts of~\eqref{equ:hF_1FaF} vanish.
Therefore we can assume that $z$ is a regular point of $\AFld$.
Then there are local coordinates $(x_1,\ldots,x_n)$ at $z=0\in\RRR^n$ in which $\AFld(x)=\frac{\partial}{\partial x_1}$ and $$\AFlow(x_1,\ldots,x_n,t)=(x_1+t, x_2,\ldots,x_n).$$
Then $\gdif(x_1,\ldots,x_n)=(x_1+\bfunc(x), x_2,\ldots,x_n)$, whence
\begin{multline*}
T\gdif\circ \AFld \circ \gdif^{-1}= 
\left(
\begin{matrix}
1+\bfunc'_{x_1} & \bfunc'_{x_2} & \cdots & \bfunc'_{x_n} \\
         0      &     1         &   0    &  0             \\
\cdots          &   \cdots      &  \cdots & \cdots         \\
         0      &     0         &   0    &  1
\end{matrix}
\right)
\left(
\begin{matrix}
\frac{\partial}{\partial x_1} \\ 0 \\ \cdots \\ 0
\end{matrix}
\right)
= \\ =
(1+\bfunc'_{x_1})\AFld = (1+\AFld(\bfunc))\AFld.
\end{multline*}

Suppose now that $\afunc$ is defined on all of $\Mman$ and $\dif$ is a diffeomorphism of all of $\Mman$.
Then by~\cite{Maks:Shifts} the function $\mu=1+\AFld(\afunc)\not=0$ on all of $\Mman$, whence by Theorem~\ref{th:image_of_shift_map} $\imSh{\mu\AFld}=\imShA$.

\section{Periodic shift maps}\label{sect:periodic_sh_maps}
Let $\AFld$ be a vector field, and $\ShA$ be its shift map.
The set $$\kerA=\ShA^{-1}(\id_{\Mman})$$ will be called the \emph{kernel} of $\ShA$, thus $\AFlow(x,\nu(x))\equiv x$ for all $\nu\in\kerA$.
Evidently, $0\in\kerA$.
Moreover, it is shown in~\cite[Lm.~5]{Maks:Shifts} that $\ShA(\afunc)=\ShA(\bfunc)$ iff $\afunc-\bfunc\in\funcA$.

Suppose that the set $\singA$ of singular points of $\AFld$ is nowhere dense in $\Mman$.
Then, \cite[Th.~12 \& Pr.~14]{Maks:Shifts}, $\ShA$ is a locally injective map with respect to any weak or strong topologies, and we have the following two possibilities for $\kerA$:

a){\bf~Non-periodic case:} $\kerA=\{0\}$, so $\ShA:\funcA\to\imShA$ is a bijection.

b){\bf~Periodic case:} there exists a smooth strictly positive function $$\theta:\Mman\to(0,+\infty)$$ such that $\AFlow(x,\theta(x))\equiv x$ and  $\kerA=\{n\theta\}_{n \in \ZZZ}$.

In this case $\funcA=C^{\infty}(\Mman,\RRR)$, $\ShA$ yields a bijection between $C^{\infty}(\Mman,\RRR)/\kerA$ and $\imShA$, and for every $\afunc\in C^{\infty}(\Mman,\RRR)$ we have that $$\ShA^{-1}\circ\ShA(\afunc) = \afunc + \kerA = \{\afunc + k\theta\}_{k\in\ZZZ}.$$
It also follows that every non-singular point $x$ of $\AFld$ is periodic of some period $\Per(x)$, $$\theta(x)=n_{x}\Per(x)$$ for some $n_{x}\in\NNN$, and in particular, $\theta$ is constant along orbits of $\AFld$.
We will call $\theta$ the \emph{period function} for $\ShA$.

\begin{lemma}\label{lm:change_period}
Suppose that the shift map $\ShA$ of $\AFld$ is periodic and let $\theta$ be its period function.
Let also $\mu:\Mman\to(0,+\infty)$ be any smooth strictly positive function.
Put $\BFld=\mu\AFld$.
Then the shift map $\ShB$ of $\BFld$ is also periodic, and its period function is 
\begin{equation}\label{equ:theta_prime_general}
\bar\theta(x) \;\stackrel{\eqref{equ:xi_inv}}{=\!=}\; \xi^{-1}(\theta)(x) \;=\; \beta(x,\theta(x))  \;=\;
\int\limits_{0}^{\theta(x)} \frac{d\tau}{\mufunc(\AFlow(x,\tau))}.
\end{equation}
If $\mu$ is constant along orbits of $\AFld$, then the last formula reduces to the following one:
\begin{equation}\label{equ:theta_prime_const}
\bar\theta=\frac{\theta}{\mu}.
\end{equation}
In particular, for the vector field $\BFld=\theta\AFld$ its period function is equal to $\bar\theta\equiv1$.
\end{lemma}
\begin{proof}
Let $\BFlow:\Mman\times\RRR\to\Mman$ be the flow of $\BFld$.
We have to show that $\BFlow(x,\bar\theta(x))\equiv x$ for all $x\in\Mman$:
\begin{equation}\label{equ:x_BFxt}
\BFlow(x,\bar\theta(x)) \stackrel{\eqref{equ:theta_prime_general}}{=\!=\!=} 
\BFlow\bigl(x,\bfunc(x,\theta(x))\bigr) = 
\AFlow(x,\theta(x)) \equiv x.
\end{equation}
Since $\theta$ is the \emph{minimal} positive function for which $\AFlow(x,\theta(x))\equiv x$ and $\mu>0$, it follows from~\eqref{equ:theta_prime_general} that so is $\bar\theta$ is also the minimal positive function for which~\eqref{equ:x_BFxt} holds true.
Hence $\bar\theta$ is the period function for the shift map of $\BFld$.

Let us prove~\eqref{equ:theta_prime_const}.
Since $\mu$ is constant along orbits of $\AFld$, we have that $\mu(\AFlow(x,\tau))=\mu(x)$, whence
$$
\bar\theta(x)=\bfunc(x,\theta(x)) = 
\int\limits_{0}^{\theta(x)} \frac{d\tau}{\mu(\AFlow(x,\tau))}  =
\int\limits_{0}^{\theta(x)} \frac{d\tau}{\mu(x)} = \frac{\theta(x)}{\mu(x)}.
$$
Lemma is proved.
\end{proof}

\subsection{Circle actions}
Regard $S^1$ as the group $U(1)$ of complex numbers with norm $1$, and let $\exp:\RRR\to S^1$ be the exponential map defined by $\exp(t)=e^{2\pi i t}$.

Let $\sact:\Mman\times S^1\to \Mman$ be a smooth action of $S^1$ on $\Mman$.
Then it yields a smooth $\RRR$-cation (or a flow) $\ract:\Mman\times\RRR\to\Mman$ given by 
\begin{equation}\label{equ:ract}
\ract(x,t) = \sact(x,\exp(t)).
\end{equation}
Moreover $\ract$ is generated by the following vector field
$$
\fldact(x) = \left.\frac{\partial\ract(x,t)}{\partial t}\right|_{t=0}.
$$
Evidently, any of $\sact$, $\ract$, and $\fldact$ determines two others.
In particular, a flow $\ract$ on $\Mman$ is of the form~\eqref{equ:ract} for some smooth circle action $\sact$ on $\Mman$ if and only if $\ract_{1}=\id_{\Mman}$, i.e. $\ract(x,1)\equiv x$ for all $x\in\Mman$.

In other words, the shift map of $\ract$ is periodic and its period function is the constant function $\theta\equiv 1$.

As a consequence of Lemma~\ref{lm:change_period} we get the following:
\begin{cor}\label{cor:persh--cicrleact}
Let $\AFld$ be a smooth vector field on $\Mman$ and $$\theta:\Mman\to(0,+\infty)$$ be a smooth strictly positive function.
Then the following conditions are equivalent:
\begin{enumerate}
 \item[\rm(a)] 
the vector field $\BFld=\theta\AFld$ yields a smooth circle action, i.e. $\BFlow(x,1)=x$ for all $x\in\Mman$;
 \item[\rm(b)] 
the shift map $\ShA$ of $\AFld$ is periodic and $\theta$ is its period function, i.e. $\AFlow(x,\theta(x))\equiv x$ for all $x\in\Mman$.
\end{enumerate}
\end{cor}

%
%

\begin{cor}\label{cor:j1_of_per_map}
Suppose that the shift map $\ShA$ of $\AFld$ is periodic and let $z\in\Mman$ be a singular point of $\AFld$.
Then there are $k,l\geq0$ such that $2k+l=\dim\Mman$, non-zero numbers $A_1,\ldots,A_k\in\RRR\setminus\{0\}$, local coordinates $(x_1,y_1,\ldots,x_k,y_k,t_1,\ldots,t_l)$ at $z=0\in\RRR^{2k+l}$, and in which the linear part of $\AFld$ at $0$ is given by
$$
\begin{array}{rcl}
j^1_{0}\AFld(x_1,y_1,\ldots,x_k,y_k,t_1,\ldots,t_l) &\!\!\!=&\!\!\! 
- A_1 y_1 \dfrac{\partial }{\partial x_1} + A_1 x_1\dfrac{\partial }{\partial y_1} + \cdots \\ [3mm]
&\!\!\! &\!\!\! -A_k y_k \dfrac{\partial }{\partial x_k} + A_k x_k\dfrac{\partial }{\partial y_k}.
\end{array}
$$
\end{cor}
\begin{proof}
Let $\theta$ be the period function for $\AFld$ and $\BFld = \theta\AFld$.
Since $\theta>0$, it follows that $\singA=\singB$ and for every $z\in\singA$ we have that 
$$j^1_{z}\BFld = \theta(z)\cdot  j^1_{z}\AFld.$$
Therefore it suffices to prove our statement for $\BFld$.

By Corollary~\ref{cor:persh--cicrleact} $\BFld$ yields a circle action, i.e. $\BFlow_{1}=\id_{\Mman}$, where $\BFlow$ is the flow of $\BFld$.
Then $\BFlow$ yields a linear flow $T_{z}\BFlow_t$ on the tangent space $T_z\Mman$ such that $T_{z}\BFlow_1=\id$.
In other words we obtain a linear action (i.e. representation) of the circle group $U(1)$ in the finite-dimensional vector space $T_z\Mman$.
Now the result follows from standard theorems about presentations of $U(1)$.
\end{proof}

\begin{rem}\rm
Suppose that in Corollary~\ref{cor:j1_of_per_map} \ $\dim\Mman=2$.
Then we can choose local coordinates $(x,y)$ at $z=0\in\RRR^2$ in which  
$$j^1_{0}\AFld(x,y) = - y \frac{\partial }{\partial x} +  x\frac{\partial }{\partial y}.$$
For this case the normal forms of such vector fields are described in~\cite{Takens:AIF:1973}.
\end{rem}

\end{document}